\documentclass[12pt, leqno]{amsart}

\usepackage{setspace}
\usepackage{graphicx}
\usepackage{amsfonts, amsmath, oldgerm, amssymb, amscd, bbm, amsthm}
\usepackage{dsfont}
\usepackage[T1]{fontenc}

\newtheorem{theorem}{Theorem}[section]
\newtheorem{lemma}[theorem]{Lemma}

\newtheorem{corollary}[theorem]{Corollary}

\newtheorem{proposition}[theorem]{Proposition}

\newcommand{\Z}{\mbox{$\mathbb Z$}}
\newcommand{\C}{\mbox{$\mathbb C$}}
\newcommand{\Q}{\mbox{$\mathbb Q$}}
\newcommand{\A}{\mbox{$\mathbb A$}}
\newcommand{\ol}{\overline}
\newcommand{\p}{\mbox{$\mathbb P$}}
\newcommand{\sk}{\ol \kappa}
\newcommand{\s}{\widetilde{S}}

\newcommand{\rar}{\rightarrow}

\title[rational cuspidal curves on $\Q$-homology projective planes]{A note on rational cuspidal curves on $\Q$-homology projective planes 
}
\author{R.V. Gurjar, DongSeon Hwang and Sagar Kolte}

\begin{document}

\begin{abstract} 
We generalize results by Wakabayashi and Orevkov  about rational cuspidal curves on the projective plane  to that  on $\Q$-homology projective planes. It turns out that the result is exactly the same as the projective plane case under suitable assumptions.  We also provide examples which demonstrate sharpness of the results. The ambient surface is singular in these examples. 
\end{abstract}

\maketitle

\section{\bf Introduction} 
The study of rational cuspidal plane curves has had a long and interesting history.  One approach is to apply the theory of open surfaces to the complement $\p^2-C$. In particular, Wakabayashi proved some sharp bounds on the number of cusps of $C$ in terms of the log Kodaira dimension of $\p^2-C$. 

\begin{theorem}$($\cite{waka78}$)$
Let $C$ be a rational cuspidal plane curve. Then, we have the following:
	\begin{enumerate}
		\item If $\sk(\p^2-C) = -\infty$, then $C$ has at most one cusp. 
		\item If $\sk(\p^2-C) \leq 1$, then $C$ has at most two cusps.
	\end{enumerate}
\end{theorem}

Recall that for a quasi-projective variety $X$ the log Kodaira dimension $\sk(X)$ is defined to be the Kodaira dimension $\kappa(\widetilde{X}, K_{\widetilde{X}} +D)$ of the pair $(\widetilde{X}, D)$ where $\widetilde{X}$ is a projectivization  of $X$ such that the boundary $D := \widetilde{X} -X$ is  a divisor with simple normal crossings.\\

Orevkov showed that the log Kodaira dimension of $\p^2-C$ is always non-zero.

\begin{theorem}$($\cite[Theorem B(c)]{orevkov02}$)$ 
Let $C$ be a rational cuspidal plane curve. Then, $\sk(\p^2-C) \neq 0$.
\end{theorem}
 
The aim of this note is to generalize these results to the case where the ambient surface is a $\Q$-homology projective plane instead of the usual projective plane. Here, a \emph{\Q-homology projective plane} is defined as a normal projective surface $S$ with quotient singularities having the same rational homology as $\p^2$. It turns out that the generalized results give us the same conclusions as the results of Wakabayashi and Orevkov under the assumption that the cuspidal curve $C$ does not pass through the singularities of $S$. \\

\begin{theorem}\label{main}
Let $C$ be a rational cuspidal curve on a $\Q$-homology projective plane $S$. Assume that $C$ is contained in the smooth locus $S^0$ of $S$. Then, we have the following:
	\begin{enumerate}
		\item If $\sk(S-C)= -\infty$, then $C$ has at most one cusp.   
		\item If $\sk(S-C) \leq 1$, then $C$ has at most two cusps.  
		\item $\sk(S^0-C) \neq 0 $.
	\end{enumerate} 
\end{theorem}

The above inequalities are sharp. See Examples in Section \ref{example}. Example \ref{ExB} also shows that we can have $\sk(S-C) = 0 $. Thus, the assumption that $C$ is contained in the smooth locus of $S$ is necessary to have a right generalization of the results on plane cuspidal curves. \\

The study of rational cuspidal curves on $\Q$-homology projective planes are  motivated by a problem posed by Koll\'ar asking the classification of pairs $(S, C)$ where  $C$ is a rational cuspidal curve on a $\Q$-homology projective plane $S$(\cite[Problem 33]{kollar08}).\\

Note that $S$ in Theorem \ref{main} is necessarily rational. Indeed, 
 since $S$ is of Picard number one, the complement $X:=S-C$ is $\Q$-acyclic, i.e., a $\Q$-homology plane. To see this note that topologically $C$ is isomorphic to $\p^1$, now using the long-exact sequence of homology for pairs we can see that $X$ is a $\Q$-homology plane. $\Q$-homology planes have been studied extensively. In particular,  R.V. Gurjar, C.R. Pradeep and A.R. Shastri have proved that every $\Q$-homology plane with quotient singularities is rational (\cite{pradeep97} and \cite{gurjar97}). We also note that, by a result of T. Fujita, \Q-homology planes are affine (cf. \cite[Theorem 1.1]{palka13}).  \\

\section{\bf Notation and Preliminaries}

Let $\s$ be a smooth projective surface over the field $\mathbb{C}$ of complex numbers. A divisor on $\s$ is a $\mathbb{Z}$-linear combination of irreducible curves on $\s$. Similarly, a $\Q$-divisor on $\s$ is defined to be a $\Q$-linear combination of irreducible curves on $\s$.  A $\Q$-divisor $D$ is said to be \emph{nef}  if $D.C \geq 0$ for every irreducible curve $C$ 
on $\s$. A $\Q$-divisor $D$ is said to be \emph{pseudo-effective} if $D.H \geq 0$ for any nef $\Q$-divisor $H$. 
Clearly, every effective divisor is pseudo-effective.  For any (possibly reducible) curve $\Delta$ on $\s$, by a \emph{component} of $\Delta$, we mean an irreducible component of $\Delta$.\\

We recall some terminology from \cite[Section 6]{fujita82}. For an irreducible 
component $Y$ of $D$, we denote $Y.(D-Y)$ by $\beta_D(Y)$. This is called the \emph{branching number} 
of $Y$ in $D$. The component $Y$ is called a \emph{tip} of $D$ if $\beta_D(Y)=1$. A sequence $C_1,...,C_r$ of irreducible 
components of $D$ is called a \emph{rational twig} $T$ of $D$ if each $C_i$ is a non-singular rational  curve, 
$\beta_D(C_1)=1$, $\beta_D(C_j)=2$ 
and $C_{j-1} \cdot C_j=1$ for $2 \leq j \leq r$. $C_1$ is called the \emph{tip} of $T$. Since $\beta_D(C_r)=2$, there is a component $Y$ 
of $D$, not in $T$, such that $C_{r} \cdot Y=1$. If $Y$ is a rational tip of $D$, then $T'= T + Y$ is a connected component of 
$D$ and will be called a \emph{rational club} of $D$. A component $Y$ of $D$ such that $\beta_D(Y)=0$ is also called a \emph{rational club} of $D$. 
When the above $Y$ is rational and $\beta_D(Y)=2$, $T'$ is again a rational twig of $D$. If $\beta_D(Y) \geq 3$ or $Y$ is 
non-rational then $T$ is called a \emph{maximal rational twig} of $D$ and $Y$ is called the \emph{branching component} of $T$. For an effective $\Q$-divisor $D$, we denote by $\Q(D)$   the $\Q$-vector space generated   by the irreducible components of $D$.\\

If $T$ is a rational twig of $D$ which is a contractible rational twig of $D$, the element $N \in \Q(T)$ such that 
$N \cdot C_1=-1$ and $N \cdot C_j=0$ for $j \geq 2$ is called the \emph{bark} of $T$. If $T'=C_1+...+C_r+Y$ is a contractible 
rational club of $D$, the bark of $T'$ is define to be the $\Q$-divisor $N'$ in $\Q(T')$ such that 
$N' \cdot C_1=N' \cdot Y=-1$ and $N' \cdot C_j=0$ for $2 \leq j \leq r$. For an isolated rational 
normal curve $Y$ with $Y^2<0$, its bark is defined 
to be $2(-Y^2)^{-1}Y$. If all the rational clubs and maximal twigs of $D$ are contractible, then the sum of their 
barks is denoted by Bk $D$ and is called the \emph{bark} of $D$. For the definition of bark of a non-linear tree of smooth rational curves, see \cite{fujita82}.\\

The next results originally due to Fujita \cite{fujita82} are useful in this work.

\begin{lemma}$($\cite[Corollary 6.14]{fujita82}$)$\label{contractible}
Let $\s$ be a smooth projective surface and $D$ be an effective reduced divisor on it.  If $D$ has a rational twig or a club which is not contractible, then $\kappa(\s, K_{\s}+D)= \ol \kappa(\s - D)= -\infty$.
\end{lemma}

\begin{lemma}\label{Fujita1}$($\cite[Lemma 6.20]{fujita82}, \cite[Lemma 1]{gurjar95}$)$  
Let $\s$ be a smooth projective surface and $D$ be an effective reduced SNC divisor on it. 
Assume that $\ol \kappa(\s-D) \geq 0$ and $N \neq$ Bk $D$ where $N$ is the negative part of the Zariski decomposition of $K_{\s} + D$. Then,  there exists 
a component $L$ of $N$ which is a $(-1)$-curve, not contained in $D$, satisfying 
one of the following conditions:

\begin{enumerate}  
\item $D \cdot L=0$, i.e. $D \cap L= \emptyset$
\item $D \cdot L=1$ and $L$ meets a twig of $D$.
\item $D \cdot L=2$ and $L$ meets exactly two components of $D$, one of which is an irreducible component of a twig of 
$D$ and the other one is a tip of a rational club of $D$.
\end{enumerate} 

Furthermore, $\sk(S-D-L)=\ol \kappa(S-D)$.
\end{lemma}

\begin{lemma}\label{Fujita2} $($\cite[Corollary 8.8]{fujita82}, \cite[Lemma 4]{gurjar12}$)$ 
Let $\s$ be a smooth projective surface and $D$ be an effective reduced SNC divisor on it. Assume that $\ol\kappa(\s-D)=0$. Assume also that every $(-1)$-curve in $D$ meets at least three other 
components of $D$. If Bk $D$ = $N$, then any connected component of $D$ is one
of the following.
\begin{enumerate}

 \item A minimal resolution of a quotient singular point.
 \item A tree of $\mathbb{P}^1$'s with exactly two branching components such that the
branching components are connected by a (possibly empty) linear chain of
$\mathbb{P}^1$'s and each branching component meets exactly two other (-2)-curves
which are tips of $D$.
 \item A simple loop of $\mathbb{P}^1$'s.
 \item A tree of $\mathbb{P}^1$'s with a unique branching component which meets three
linear trees defining cyclic quotient singular points at one of their end
points. Furthermore, the absolute values $d_1,d_2,d_3$ of the discriminants of the
three trees satisfy $\Sigma 1/d_i=1$. 
 \item A tree of five $\mathbb{P}^1$'s with a unique branching component  which
intersects the other four curves transversally in one point each, and such that
the four curves are all (-2)-curves.
 \item A smooth elliptic curve.
\end{enumerate}

\end{lemma}

From now on, let $C$ be a rational cuspidal curve contained in the smooth locus of a $\Q$-homology projective plane  $S$. Let $\s\rar S$ be a resolution of singularities of both $S$ and $C$ such that the total transform $D$ of $C$ is a simple normal crossing divisor. Let $C'$ be the proper transform of $C$. We always assume that the map is minimal, i.e., the Picard number of $\s$ is the least possible. The map $\s\rar S$  factors 
through the minimal resolution $S' \rightarrow S$ of singularities of $S$. 
Let $E$ be the reduced exceptional divisor for the resolution of singularities of $S$.
Since $S$ has only quotient singularities,  $E$ is a simple normal crossing divisor. The smooth locus of $S$ is denoted by $S^0$. We will repeatedly use the well-known properties of singular fibers of a $\p^1$-fibration on a smooth surface 
\cite[Chapter I, Lemma 4.4.1]{miya81}.

\begin{lemma}\label{1sing}
With the above notation, if $\ol \kappa(S'-C) = -\infty$, then $C$ has at most  one  singular point. 
\end{lemma}

\begin{proof}
Let $\widetilde{S} \rightarrow S'$ be a resolution of singularities on $C$ and let $D$ be the total transform of $C$. Since $D$ is connected, by a result of Miyanishi-Sugie and Russell(see \cite[Theorem I.3.13]{miya81}), there is an $\A^1$- fibration 
on $\widetilde{S}-D$. Let $h: \widetilde{S}' \rightarrow \widetilde{S}$ be a birational morphism 
such that the $\A^1$-fibration extends to a $\p^1$-fibration on $\widetilde{S}'$, and $D'$ be the proper transform of $D$ under $h$. \\

Assume that $C$ has at least two cusps.  Then,  $D$ has at least two $(-1)$-curves with branching number $3$. Indeed, those $(-1)$-curves arise from the final blow-ups in order to resolve the cusps of $C$ and make $D$ an SNC divisor. Note that $C$ cannot be a section of the $\p^1$-fibration on $\widetilde{S}'$. Indeed,  as it has a $(-1)$-curve meeting it transversely, this $(-1)$-curve will therefore be reduced and have branching number 2 in a fiber of a $\p^1$-fibration. This is impossible.
Hence at least one of the singular fibers of the $\p^1$-fibration on $\widetilde{S}'$ 
has a $(-1)$-curve with branching number $3$ in the fiber. This is not possible by well-known properties of singular fibers of a 
$\p^1$-fibration on a smooth projective surface. \\
\end{proof}

The following lemma is useful in proving  Theorem \ref{main} (3).

\begin{lemma}\label{inf}
If $\sk(\s-D)=-\infty$, then $\sk(\s-D-E)=-\infty$.
\end{lemma}
\begin{proof}
Before we begin the proof, we describe the structure of $D$ when $C$ is uni-cuspidal. There is a unique $(-1)$-curve in 
$D$, which we denote by $H$, such that its branching number is three. The proper transform $C'$ of $C$ is one of the 
branches of $H$. Let $T_2$ and $T_3$ denote the other two branches of $H$. It is easy to see that  one of these two 
(say $T_2$) is a linear chain of rational curves and the other  branch, say $T_3$, is a tree of rational curves such that every 
branch of $T_3$ has branching number three(\cite[Proposition 3.2]{orevkov02}).  We will use this description implicitly in the arguments that follow.\\

By a result of Koras and Palka(\cite{koras13}), $\sk(\s-D-E) < 2$. Suppose that $\sk(\s-D-E) \geq 0$. 
By \cite[Theorem 5.1]{palka12}, $S-C$ has an untwisted  $\C^*$-fibration $f$ over $\A^1$ such that the 
fiber at infinity can be assumed to be smooth and $f$ has two singular fibers. One of the singular fibers (say $F_1$) is a non-reduced punctured affine line $\C^*$   
and the other (say $F_2$) consists of two $\C$'s meeting in a cyclic singular point. Furthermore, both $\C$'s have multiplicity at least 
two. The $\C^*$-fibration  is naturally extended to $\s-D$ and the extended action is also denote by $f$. We will derive a contradiction by showing  that such a $\C^*$-fibration cannot exist on $\s-D$.\\

We claim that $f$ is again extended to  a $\p^1$-fibration on $\s$. If not, then $f$ has a base point on $D$. Let $D'$ be the total transform of $D$ under 
the resolution of base locus. It is clear that the $(-1)$-curve resulting from the last blow up will be one of the horizontal  components of $D'$ to $f$ which has branching number at most two. This is a contradiction to the description of the $\C^*$-fibration by Palka (\cite[Figure 3, pp.450]{palka12}) where each  horizontal component has branching number three. Indeed, if one of the horizontal components do not have branching number 3, then $F_1$ would be reduced, contrary to the assertion ($\mu,\tilde{\mu} \geq 2$) in \cite[Theorem 5.1]{palka12}. Thus $f$ has no base points on $D$.\\

The extended $\p^1$-fibration on $\s$ is also denoted by $f$.   A fiber of $f$ is contained in 
$D$. We denote this fiber by $R$. The fiber $R$ is reducible because $D$ does not contain an irreducible component 
with self intersection 0 and branching number two. Clearly $R$ contains a $(-1)$-curve. Furthermore, the only components 
of $D$ that can possibly have self intersection $-1$ are $C'$ and $H$.\\

We claim that $R= C' \cup H$. First note that $H$ is contained in $R$. Indeed, if $H$ is not contained in $R$, it 
forces $(C')^2=-1$ and $C' \subset R$. But $C'$ is a tip of $D$ and $R$ is a rivet of $f$. Thus, the unique component of 
$D$ which is adjacent to $C'$ (that is $H$) has to be in $R$, a contradiction. Thus, $H \subset R$. Now at least one of the irreducible components of $D$ adjacent to $H$ must be horizontal to the fibration, since otherwise we will have a $(-1)$-curve with branching number three in a singular fiber of a 
$\p^1$ fibration. Next, if one of the irreducible components of $D$ adjacent to $H$ is a cross-section then $H$ is reduced 
in $R$. Therefore $H$ cannot have branching number two in $R$. This shows that $H$ meets both sections of $f$ 
in $D$. This forces $C'^2=-1$ and $R=H \cup C'$. This proves our claim.\\

Using the claim and the fact that $T_2$ is a linear chain of rational curves, we see that the branching number of one 
of the horizontal components of $f$ is two. By \cite[Theorem 5.1]{palka12}, $\sk(\s-D) \geq 0$, a contradiction. \end{proof}

\begin{corollary}\label{(-1)}
If $C'$ is a $(-1)$-curve, then  $\sk(\s-D-E)=-\infty$.
\end{corollary}

\begin{proof}
If we blow down $C'$, the divisor $D$ maps to a divisor  with a non-contractible twig, which would imply, by Lemma \ref{contractible}, that $\sk(\s-D)=-\infty$. Now the result follows by Lemma \ref{inf}.
\end{proof}

Finally, we record a well-known observation about $\C^*$-bundles:
\begin{proposition}\label{bundle}
Let $f: X \rightarrow \C^*$ be a morphism with every scheme-theoretic fiber $\C^*$. Then $\sk(X)=0$.
\end{proposition}

\begin{proof}
If $f$ is Zariski locally trivial (with respect to the base) then it is a principal bundle associated to a line bundle 
on $\C^*$. Since every algebraic vector bundle on $\C^*$ is trivial, the total space of this principal bundle 
is a product $\C^* \times \C^*$. Therefore $\sk(X)=0$.\\

If $f$ is not Zariski locally trivial, then the pullback of $f$ to a twofold finite etale cover of the base $\C^*$ is 
Zariski locally trivial. Hence a twofold finite etale cover of $X$ has $\sk=0$. But $\sk$ is preserved under finite etale covers by
a result of Iitaka. This proves the result.

\end{proof}

\section{Proof of Theorem \ref{main} (1),(2)}

Theorem \ref{main} (1) directly follows from Lemma \ref{1sing} by taking $S'$ as a minimal resolution of $S$. Now we prove Theorem \ref{main} (2). 
 
\subsection{The case $\ol \kappa(\widetilde{S}-D) =0$}   

If $C$ has at least three cusps, then $C'$ has more than 2 branches in $D$. Furthermore, each branching component is attached to $C'$ by means of a $(-1)$-curve each having branching number 3 in $D$. 
By Lemma \ref{Fujita2}, Bk $D$ $\neq N$. Then,  there is a $(-1)$-curve not contained in $D$ satisfying one of the 
three conditions in Lemma \ref{Fujita1}. By blowing down this curve and any other $(-1)$-curves until Bk $D = N$, $D$ maps to one of 
the divisors listed in Lemma \ref{Fujita2}. Here we denote the image of $D$ under the blowdown map again by $D$.   In our case, only the cases (2) and (4) in Lemma \ref{Fujita2} is possible (as the intersection matrix of $D$ has a non-negative eigenvalue, the image of $D$ under the blow-downs of $(-1)$-curves also has a non-negative eigenvalue hence (1) of \ref{Fujita2} is not possible).  The image of $C'$ under the blowing downs have branching number at least $3$ as the branches of $C'$ are attached to it by $(-1)$-curves. Hence, by Lemma $\ref{Fujita2}$ (2) and (4), at least two of the branches of the image of $C'$ will have to be contractible linear chains of rational curves whose tips meet the image of $C'$. This is not possible as the branches of $C'$ are attached to $C'$ by a $(-1)$-curve which is not a tip of a rational twig.\\

\subsection{The case $\ol \kappa(\widetilde{S}-D) =1$} 
In this case, because $X$ is affine, by \cite[Chapter 3, Theorem 1.7.1]{miyanishi01} originally due to Y. Kawamata, there is a $\C^*$-fibration on $\s-D$. \\

Assume that $C$ has at least three cusps. Then,  $D$ has at least three $(-1)$-curves each of which is a branching component with branching number three. Furthermore, none of them are adjacent to each other.\\

The $\C^*$-fibration extends to a $\p^1$-fibration $f$ on $\s$. Indeed, if otherwise, there are at most two base points, hence, there is at least one $(-1)$-curve in $D$  on which there is no base point. This $(-1)$-curve will be in a fiber of $f$ since it is neither a section nor a $2$-section. Furthermore, $C'$ cannot be a section of the $\p^1$-fibration $f$ on $\widetilde{S}$. Indeed, if $C'$ were a section, then one of the $(-1)$-curves adjacent to $C'$  which is in a fiber will be reduced and have branching number 2 in the fiber, a contradiction.  Thus, after resolution of the base locus, this $(-1)$-curve is in a fiber of the $\p^1$-fibration and has branching number $3$ in the fiber. This is  not possible.\\

Now that the $\C^*$-fibration extends to a $\p^1$-fibration $f$ on $\s$, we claim that $C'$ is not a 2-section of $f$. Indeed, if otherwise, we would have three $(-1)$-curves in three singular fibers meeting $C'$ once with multiplicity 2. Thus there would be a 2-1 map from $C'$ to a rational curve with 3 points of ramification. This is impossible.\\

Furthermore, none of the $(-1)$-curves in $D$ can be contained in a fiber: If a $(-1)$-curve $l$ from $D$ is contained in a fiber $F$ of the $\p^1$-fibration $f$ then $l$ cannot have branching number three in $F$. Thus at least one of the curves in $D$ adjacent to $l$ is horizontal to $f$. But this makes $l$ a reduced $(-1)$-curve in $F$. Hence $l$ cannot have branching number two in $F$. Thus, two curves adjacent to $l$ are horizontal to $f$. None of them are $(-1)$-curves, hence the other $(-1)$-curves in $D$ are in the fibers of $f$ as there are only two components of $D$ which are horizontal to $f$.   Let $l'$ be the other $(-1)$-curve in $D$ which is contained in a fiber $F'$ of $f$. By a similar reasoning, two of the branching components of $l'$ are horizontal to $f$. Thus, $D$ contains a loop. This is not possible as $D$ is a tree of rational curves. Hence none of the $(-1)$-curves are in fibers of $f$ thus all the $(-1)$-curves in $D$ are horizontal to $f$. But there are exactly two components of $D$ which are horizontal to $f$. Hence $D$ cannot have more than two $(-1)$-curves which is a contradiction.

\section{\bf Proof of Theorem 1.3 (3)}

We prove Theorem \ref{main} (3).  Assume that $\sk(\s-D-E) = 0$. By \cite[Theorem 6.1]{palka12}, if $\sk(\s-D-E)=0$, then there is a twisted $\C^*$-fibration on $\s-D-E$ or $S-C$ is an exceptional $\Q$-homology plane as described in \cite{palka11}. Moreover, the boundary of an exceptional $\Q$-homology plane is a fork with branching $(-1)$-curve and three maximal twigs which are either $[2]$ and two $[2,2,2]$'s,  or three $[2,2]$'s (\cite[Proposition 4.4]{palka11}).  These do not occur in our case since $D$ is the SNC resolution of a cuspidal rational curve and $C'$ is not a $(-1)$-curve by Corollary \ref{(-1)}.  Now the following lemma gives a contradiction.  

\begin{lemma}\label{untwisted}
If $ \sk(\s-D-E) =0$ or $1$,   then every $\C^*$-fibration on $\s-D-E$ is untwisted.
\end{lemma}

\begin{proof}
The $\C^*$-fibration $f$ on $\s-D-E$ can be extended to a $\C^*$-fibration on $\s-D$. That is to say, 
$E$ is contained in fibers of $f$. Indeed, if otherwise, an irreducible component 
of $E$ is horizontal for $f$. Thus, $\widetilde{S} - D$ contains a cylinder-like open set. Hence,  by (\cite{miya81}, Chapter 1, Theorem 3.13) that 
$\sk(\s-D)=-\infty$.  By Lemma \ref{inf}, this implies $\sk(\s-D-E)=-\infty$, a contradiction. \\

Since  $ \sk(\s-D-E) < 2$, we have $\sk(\s-D) <2$, thus $C$ has at most two cusps by Theorem 1.3 (1) and (2).\\

Case 1. C is uni-cuspidal.\\
We use the notation in Lemma \ref{inf} in which  a description of $D$, when $C$ is uni-cuspidal, is given.\\

Assume that $f$ is a twisted $\C^*$-fibration on $\s-D-E$. By  \cite[Lemma 2.10]{miya91}, the base of the $\C^*$-fibration $f$ 
is $\A^1$. Furthermore, all singular fibers of $f$ are irreducible and exactly one of them is an affine line (under the reduced  structure). By \cite[Lemma 4.1]{palka12}, the fiber at infinity is of the type $[2,1,2]$.\\

We claim that $f$ has no base points on $D$. If $f$ has a base point on $D$, it has 
to be unique. Indeed, if $f$ has two base points on $D$, then $f$ is an untwisted fibration. Furthermore, the base point of 
$f$ will have to lie on $H$ as $H$ is the unique $(-1)$-curve in $D$ with branching number three.  Let $D'$ be the total 
transform of $D$ under the resolution of base locus. As the base point is on $H$, $D'$ has a unique $(-1)$-curve which 
is horizontal to $f$. But the fiber at infinity must also contain a $(-1)$-curve, a contradiction.  \\

Since $f$ has no base points on $D$,  $H$ is the unique $(-1)$-curve in $D$ and is contained in the fiber at 
infinity which is of the type $[2,1,2]$.  Furthermore, $C'$ is a $(-2)$-curve contained in the fiber at infinity and the linear chain $T_2$ is an irreducible $(-2)$-curve. Let $\theta$ denote the 2-section of $f$ which is an irreducible component of $D$ adjacent to $H$. As $D$ can be blown down to $C$, it is easy to deduce that $\theta$ is a $(-3)$-curve.\\

Recall that $\theta$ has branching number at most three. First, we show that $\theta$ has at most two branches.  Assume that   $\theta$ has branching number three. By $M_1$ and $M_2$ we 
denote the two branches of $\theta$ in $D$ other than $H$. Note that $M_1$ and $M_2$ contain no $(-1)$-curves. If $M_1$ and $M_2$ are in separate fibers, then both fibers will contain an affine line $\A^1$ for the fibration on $\s-D$, which is a contradiction to \cite[Lemma 2.10]{miya91}. Hence $M_1$ and $M_2$ lie in the same singular fiber $F_0$ of $f$. By \cite[Lemma 2.10]{miya91}, the fiber $F_0$ is linear with a unique $(-1)$-component separating $M_1$ and $M_2$. We see from \cite[Lemma 2.9]{miya91} that $E$ lies on a singular fiber of $f$ other than $F_0$. 
By well-known properties of a singular fiber of a $\p^1$-fibration,  $M_1$ and $M_2$  have the same determinant (up to sign).   On the other hand, after contraction of the $(-1)$-curve and one $(-2)$-curve in the fiber
at infinity, $\theta$ becomes a $(-1)$-curve and together with $M_1$ and $M_2$ it
contracts to a smooth point on the original $\Q$-homology projective plane $S$. This
implies that the intersection matrix of the union of $M_1$, $M_2$ and the $(-1)$-curve is
unimodular. By \cite[Lemma 6.1]{gs89},   the determinants of $M_1$ and $M_2$ are
pairwise coprime, which is a contradiction. \\ 

Now we show that $\theta$ has only one branch, that is, $H$. In other words, $T_3$ is irreducible. Let $M_1$ be the branch of $\theta$ other than  $H$. Thus there is an open subset of $\s-D-E$ which is a $\C^*$-bundle over $\C^*$, hence, by Proposition \ref{bundle}, $\sk(\s-D-E)=0$. Let $F_{M_1}$ be the fiber containing $M_1$. This fiber meets $\theta$ in only one point. There is a $(-1)$-curve in $F_{M_1}$ which meets $M_1$. We blow this curve down and the subsequent $(-1)$-curves until the image of $D$ is minimally SNC. By Lemma \ref{Fujita2}, $F_{M_1}$ maps to a fiber of the type $[2,1,2]$ and $\theta$ maps to a $(-3)$-curve as it is not affected during the process of blowing down.\\

To see that $\theta$ is not affected during the blowing down, note that there is a unique $(-1)$-curve in $F_{M_1}$ by  \cite[Lemma 2.10]{miya91}. By Lemma 2.2, this $(-1)$-curve meets a tip of $E$ which is a cyclic quotient singularity. If we blow down this $(-1)$-curve and continue to blow down until Bk $D=N$, then by Lemma \ref{Fujita2} we see that the image of $F_{M_1}$ is [2,1,2]. Here the $(-1)$-curve is the image of a component $U$ of $D$ and the two $(-2)$-curves are images of curves in $D$ whose self intersection is also $-2$ as no curve adjacent to them could have been blown down. But $D$ is obtained by resolving a cusp of a rational curve SNC minimally, so $D-C'$ can be blown down to a smooth point. Hence the image of $U$ cannot have two adjacent $(-2)$-curves. 
Hence the self-intersection of $\theta$ remains unaffected under the blow downs.\\

If we blow down the singular fibers of the resulting $\p^1$-fibration to irreducible curves, the image of $\theta$ is a smooth 2-section with self-intersection 1. By elementary properties of ruled surfaces this is not possible. Thus, $\theta$ has only one branch.\\

This shows that 
the number of singular fibers of the $\p^1$-fibration $f$ is exactly two. Hence we have an open subset of $\s-D-E$ which is a twisted 
$\C^*$-bundle over $\C^*$.  Thus, by Proposition 2.6, $\sk(\s-D-E)=0$.\\

Hence $D$ has a unique branching component, i.e. $H$. Moreover, the image 
of $H$ has three branches such that $\Sigma 1/d_i > 1$. Because $\sk(\s-D-E)=0$ we get $\sk(\s-D)=0$. This is impossible by Lemma \ref{Fujita2}.\\

Case 2. C is bi-cuspidal.\\

We begin with a description of $D$ when $C$ is bi-cuspidal. The 
branching number of $C'$ is two. It is adjacent to two $(-1)$-curves $H_1$ and $H_2$ both having branching number three. 
By $B_{11}$ and $B_{12}$ we denote the branches of $H_1$ other than $C'$. By $B_{22}$ and $B_{21}$ we denote the 
branches of $H_2$ other than $C'$. We can assume $B_{11}$ and $B_{22}$ to be linear. Both branches   of $C'$ can be blown down to smooth points.  \\

Recall that we have a twisted $\mathbb{C}^*$-fibration $f$ on $\widetilde{S}-D-E$. 
We claim that the fibration $f$ has no base point on $D$. Indeed, if otherwise, as we argued in Case 1, the base point of $f$ on $D$ have to be unique because $f$ is twisted.  Also,  the base point of $f$ has to lie on one of the $(-1)$-curves with branching number three. This forces the other $(-1)$-curve with branching number 
three to be in a singular fiber and have branching number three in the fiber. This is not possible.  \\

Now that $f$ has no base points on $D$, by \cite[Lemma 4.1]{palka12}, there is a fiber at infinity  of  type $[2,1,2]$. This is impossible by the description of $D$ above.
\end{proof}

\section{\bf Examples}\label{example}

In this section we give three examples of pairs $(S,C)$ such that $\sk(S-C)<2$.

\subsection{Example A}\label{ExA}
We start with the Hirzebruch surface $\Sigma_0$. We fix a fiber $C'$ of one of the $\p^1$-fibrations, say $f$, and a 
cross-section $H_0$ on $\Sigma_0$. 
By choosing two other fibers of $f$ we blow up suitably to produce linear chains of the type $[4,2,1,3,2,2]$ and
$[2,2,1,3]$ such that the proper transform $H$ of
$H_0$  a $(-1)$-curve. We contract the subchains $[3,2,2]$ and $[3]$ in these chains.   Let $S$ be the surface obtained by blowing down $H$ and the two linear chains $[4,2]$ and
$[2,2]$ to a smooth point on a normal surface $S$. By keeping track of the rank of the Picard group of the surface, we see that it is a $\Q$-homology projective plane. The image of $C'$ in $S$ is a unicuspidal rational curve $C$ contained
in $S^0$. Both singular points of $S$ are not rational double points. One can see that $\sk(\s-D)=-\infty$ and $\sk(\s-D-E)=-\infty$.

\subsection{Example B}\label{ExB}
We start with the Hirzebruch Surface $\Sigma_1$. Let $E$ be the minimal section with negative self-intersection. 
There exists a smooth rational curve $C$ which meets every fiber twice and is disjoint from $E$. Let $F_1$ and 
$F_2$ be the fibers that meet $C$ tangentially. Let $p_1$ and $p_2$ be the points of intersection of $F_1$ and $F_2$ with $E$ respectively. We perform elementary transforms by blowing up at $p_1$ and $p_2$ and then blowing down 
the proper transforms of $F_1$ and $F_2$. In this process $C$ acquires two cusps and $E^2=-3$. Now blow down $E$ to 
a singular point. This makes $S$ a $\Q$-homology projective plane. One can see that  $\sk(\s-D)=0$  and $\sk(\s-D-E)=1$.

\subsection{Example C}\label{ExC}
Consider a cubic $C$ on $\C\p^2$ with a single cusp. Let $T$ be the tangent line to $C$ at the cusp. Let $L$ be 
the tangent to $C$ at a smooth point of inflection. Let $p= T \cap L$, clearly $p$ is not on $C$. We blow up at $p$. 
We have a fibration with the proper transforms of $T$ and $L$ are fibers. Let $q$ be the point where the proper 
transform $L'$ of $L$ and $E$ meet. Here $E$ is the exceptional curve. We perform elementary transforms by blowing up 
at the point $q$ and blowing down the proper transform of $L$. Thus, $C$ has two cusps. One can see that  $\sk(\s-D)=1$ and $\sk(\s-D-E)=1$. As above, keeping track of the rank of the Picard group, we can conclude that $S$ is a $\Q$-homology projective plane.

\section*{Acknowledgement}
DongSeon Hwang was supported by Basic Science Research Program through the National Research Foundation of Korea(NRF) funded by the Ministry of Education(NRF-2015R1D1A1A01060179). The authors would like to thank the referee for the careful reading and suggestions.

\bibliographystyle{plain}
\bibliography{biblio}
\vspace{1em}
R.V. Gurjar\\
Tata Institute of Fundamental Research, Mumbai.\\
gurjar@math.tifr.res.in\\
Current Address:\\
Dept. of Mathematics, IIT Bombay, Powai, Mumbai 400076, India.\\
gurjar@math.iitb.ac.in\\

DongSeon Hwang\\
Department of Mathematics, Ajou University, Suwon 16499, Republic of Korea.\\
dshwang@ajou.ac.kr\\

Sagar Kolte\\
Tata Institute of Fundamental Research, Mumbai.\\
sagar@math.tifr.res.in\\
Current Address:\\
Credit Suisse, Delphi Building, Powai, Mumbai 400076, India.\\
sagar.kolte@credit-suisse.com

\end{document}